\documentclass[11pt]{amsart}
\usepackage{amsmath}
\usepackage{amsthm}
\usepackage{amsfonts}
\usepackage{amssymb}

\DeclareMathOperator{\re}{Re}
\DeclareMathOperator{\im}{Im}

  \newtheorem{theorem}{Theorem}
  \newtheorem{corollary}{Corollary}
  
  \newtheorem{lemma}{Lemma}

\begin{document}

\title[Set partitions without blocks of certain sizes]{Set partitions without blocks \\ of certain sizes}
\author{Joshua Culver}
\address{
Department of Mathematics,
Southern Utah University,
Cedar City, Utah 84720}
\email{joshua.e.culver@gmail.com}

\author{Andreas Weingartner}
\address{
Department of Mathematics,
Southern Utah University,
Cedar City, Utah 84720}
\email{weingartner@suu.edu}



\begin{abstract}
We give an asymptotic estimate for the number of partitions of a set of $n$ elements, whose block sizes avoid a given set $\mathcal{S}$ of natural numbers. As an application, we derive an estimate for the number of partitions of a set with $n$ elements, which have the property that its blocks can be combined to form subsets of any size between $1$ and $n$.
\end{abstract}

\maketitle

\section{Introduction}

Let $B_n$ be the $n$-th Bell number, that is the number of set-partitions of a set with $n$ distinct elements. 
For example, $B_3=5$ because there are five set-partitions of $\{a,b,c\}$:
$$\{\{a,b,c\}\}, \ \{\{a\},\{b,c\}\}, \ \{\{b\},\{a,c\}\}, \ \{\{c\},\{a,b\}\}, \ \{\{a\},\{b\},\{c\}\}.$$
We say that the partition $\{\{a\},\{b,c\}\}$ has the block $\{a\}$ of size $1$ and the block $\{b,c\}$ of size $2$.
Blocks of size zero (empty blocks) are not allowed.

Let $\mathcal{S} \subset \mathbb{N}$ and let 
$B_{n,\mathcal{S}}$ denote the number of partitions of a set with $n$ elements, whose block sizes are not in $\mathcal{S}$.
If $\mathcal{S}=\{1,2,\ldots, m\}$, we write
$B_{n,m}=B_{n,\mathcal{S}}$, the number of partitions of a set with $n$ elements, all of whose block sizes  are greater than $m$. 
We shall call such partitions \emph{$m$-rough}. For example, $B_{3,1}=B_{3,2}=1$. We clearly have $B_{n,0}=B_n$.
We define $B_{0,\mathcal{S}}=1$ for all $\mathcal{S}$.

Throughout this paper, we write $r=r(n)$ to denote the solution of 
\begin{equation}\label{req}
  r e^{r}=n.
\end{equation}
The function $r$ is called Lambert-$W$ function or product logarithm, and it can be approximated with the asymptotic formula (see \cite{CGH})
\begin{equation}\label{Wasymp}
r=\log n - \log \log n +\frac{\log\log n}{ \log n} + O\left(\left(\frac{\log\log n}{ \log n}\right)^2\right).
\end{equation}
Let
$$\alpha(z)=\alpha_\mathcal{S}(z)=  \sum_{k\in \mathcal{S}} \frac{z^k}{k!}.$$

Our first result is an asymptotic estimate for $B_{n,\mathcal{S}}$, derived from 
Cauchy's residue theorem and the saddle point method.
\begin{theorem}\label{thm1}
Let $\eta_1=0.1866823...$, $\eta_2=2.1555352...$ be the two real solutions of $\eta(1-\log \eta )=1/2$ 
and let $0<\delta_1 < \eta_1 < \eta_2 <\delta_2.$
We have 
$$B_{n,\mathcal{S}}= \frac{n! \exp\bigl(e^r-1-\alpha(r)\bigr)}{ r^n \sqrt{2\pi r(r+1)e^r}} \left(1+O\left(\frac{1+(\alpha'(r))^2}{e^r}\right)\right),$$
uniformly for sets $\mathcal{S}$ with $\mathcal{S} \cap [\delta_1 r , \delta_2 r] = \emptyset$.
\end{theorem}
Lemma \ref{alpha2} shows that the relative error term in Theorem \ref{thm1} approaches zero as $n\to \infty$, 
uniformly for sets $\mathcal{S}$ with $\mathcal{S} \cap [\delta_1 r , \delta_2 r] = \emptyset$.
The constant factor implied in the big-O notation depends only on the choice of the constants $\delta_1, \delta_2$, but 
does not depend on $n$ or $\mathcal{S}$.
On the other hand, if $\mathcal{S} \cap [\gamma_1 r , \gamma_2 r] \neq \emptyset$, for
constants $\eta_1<\gamma_1 < \gamma_2 < \eta_2$, then the error term in Theorem \ref{thm1} grows unbounded
as $n\to \infty$. In this case, it seems that a different method is needed to determine the asymptotic behavior of $B_{n,\mathcal{S}}$. 

The contribution to $\alpha(r)$ from $k\ge (e+\varepsilon)r$ is $o(1)$ as $n\to \infty$, if $\varepsilon>0$. Thus, 
$B_{n,\mathcal{S}}\sim B_{n,\mathcal{S'}}$ if $\mathcal{S}$ and $\mathcal{S}'$ 
avoid  $[\delta_1 r , \delta_2 r]$  and differ only on  $ ((e+\varepsilon)r,\infty)$.

If $\max \mathcal{S} \le r$, the occurrence of $\alpha'(r)$ in the error term of Theorem \ref{thm1} can be estimated 
as in equation \eqref{alphaineq}, with $i=1$, to obtain 
$ \alpha'(r) < \left(r/m\right)^{m-1} e^m $, where $m=\max \mathcal{S}$.
\begin{corollary}\label{core}
With  $\delta_1$ as in Theorem \ref{thm1} and $m=\max \mathcal{S}$, we have
$$B_{n,\mathcal{S}}= \frac{n! \exp\bigl(e^r-1-\alpha(r)\bigr)}{ r^n \sqrt{2\pi r(r+1)e^r}} \left(1+O\left(\frac{(e r/m)^{2m-2}}{e^r}\right)\right),$$
uniformly for sets $\mathcal{S}$ with $1\le m \le \delta_1 r$.
\end{corollary} 

 If $\mathcal{S}=\emptyset$, then $\alpha(r)=0$ and Theorem \ref{thm1} simplifies to the known asymptotic estimate for the Bell numbers $B_n$ (see Moser and Wyman \cite{MW}):
\begin{corollary}\label{cor0}
We have 
$$B_{n}= \frac{n! \exp\left(e^r-1\right)}{ r^n \sqrt{2\pi r(r+1)e^r}} \Bigl(1+O\left(e^{-r}\right)\Bigr).$$
\end{corollary}
Dividing the estimate in Theorem \ref{thm1} by the one in Corollary \ref{cor0} leads to the following result.
\begin{corollary}\label{cor1}
Let  $\delta_1,\delta_2$ be as in Theorem \ref{thm1}. The proportion of set partitions of $n$ objects, whose block sizes are not in $\mathcal{S}$, is
$$ \frac{B_{n,\mathcal{S}}}{B_n} = \exp\bigl(-\alpha(r) \bigr)\left(1+O\left(\frac{(\alpha'(r))^2}{e^r}\right)\right)
= \exp\bigl(-\alpha(r) \bigr)\bigl(1+o(1)\bigr) ,$$
as $n\to \infty$, uniformly for sets $\mathcal{S}$ with $\mathcal{S} \cap [\delta_1 r , \delta_2 r] = \emptyset$.
\end{corollary}

\begin{table}[ht]
  \begin{tabular}{ | c | c |c|c| }
    \hline
    $n$ & $B_{n,1}/B_n$ & $\exp(-r)$ & Rel. Error  \\ \hline  
    $2^2$ & 0.266667 &0.300542 &0.127032 \\ \hline
    $2^4$ & 0.116036 &0.128325 &0.105906  \\ \hline
    $2^6$ & 0.045716 & 0.047583 & 0.040834   \\ \hline
     $2^8$ & 0.015896 & 0.016123 & 0.014298    \\ \hline
     $2^{10}$ & 0.005122 &0.005146 & 0.004675 \\ \hline
     $2^{12}$ & 0.001573 & 0.001575 & 0.001456 \\ \hline
      $2^{14}$ & 0.000468 & 0.000468 & 0.000438  \\ \hline
     \end{tabular}
     \medskip
 
     \caption{Proportion of 1-rough set partitions: numerical examples of Corollary \ref{cor1} for  $\mathcal{S}=\{1\}$ and $\alpha(r)=r$, showing the ratio $B_{n,\mathcal{S}}/B_n$, the approximation $\exp(-\alpha(r))$ and the relative error $\exp(-\alpha(r))/(B_{n,\mathcal{S}}/B_n)-1$.
}\label{table1}
\end{table}

\begin{table}[ht]
   \begin{tabular}{ |c|c|c|c|c| }
    \hline
    $n$ &  $B_{n,2}/B_n$ & $\exp(-r-\frac{r^2}{2})$ & Rel. Error & $(1+r)^2 e^{-r}$\\ \hline  
    $2^2$  & $6.667 \cdot 10^{-2}$ & $1.459 \cdot 10^{-1}$ & 1.1886 & 1.4575 \\ \hline
    $2^4$  & $ 8.772 \cdot 10^{-3}$ &$ 1.559 \cdot 10^{-2}$& 0.7776 & 1.1962 \\ \hline
    $2^6$  & $3.185 \cdot 10^{-4}$ & $4.610 \cdot 10^{-4}$ &0.4474  & 0.7787 \\ \hline
     $2^8$  & $2.628\cdot 10^{-6}$ &  $3.222\cdot 10^{-6}$ &  0.2257 & 0.4239 \\ \hline
     $2^{10}$  & $4.356 \cdot 10^{-9}$ &  $4.805 \cdot 10^{-9}$ &  0.1033 & 0.2023 \\ \hline
     $2^{12}$ & $1.368 \cdot 10^{-12} $ & $1.428 \cdot 10^{-12}$ &  0.0438 & 0.0875 \\ \hline
      $2^{14}$  & $7.902 \cdot 10^{-17} $ & $8.040 \cdot 10^{-17}$ &  0.0175 & 0.0352 \\ \hline
     \end{tabular}
     \medskip
     \caption{Proportion of 2-rough set partitions: numerical examples of Corollary \ref{cor1} for $\mathcal{S}=\{1,2\}$ and $\alpha(r)=r+r^2/2$, showing the ratio $B_{n,\mathcal{S}}/B_n$, the approximation $\exp(-\alpha(r))$, the relative error $\exp(-\alpha(r))/(B_{n,\mathcal{S}}/B_n)-1$, and  $(\alpha'(r))^2 e^{-r}$, the relative error term in Corollary \ref{cor1}.
}\label{table1b}
\end{table}

With 1-rough partitions, $\alpha(r)=r$, so that $e^{-r}$ is the main term as well as the relative error term in Corollary \ref{cor1}.
This is consistent with Table \ref{table1}, where the relative errors are $<e^{-r}$.
Similarly, for 2-rough partitions, the relative errors in Table \ref{table1b} are $<(1+r)^2 e^{-r}$, the relative error term in Corollary \ref{cor1}.

Since $e^{-r}=r/n$ and $r\sim\log n$ by equation \eqref{Wasymp}, the proportion of 1-rough set partitions (i.e. partitions with no singletons) is asymptotic to $(\log n) /n$. 
Corollary \ref{cor2} makes that more precise.
\begin{corollary}\label{cor2}
The proportion of $1$-rough set partitions of $n$ objects is
$$ \frac{B_{n,1}}{B_n} = \frac{r}{n}\left(1+O\left(\frac{\log n}{n}\right)\right) 
= \frac{\log (n/\log n) }{n}\left(1+O\left(\frac{\log\log n}{(\log n)^2}\right)\right) .$$
\end{corollary}

The quantity $B_{n,m}$ appears in \cite{WM}. However, \cite[Prop. 2]{WM} claims that 
$B_{n,1}/B_n \sim (\log n)/(ne)$, which is false in light of Corollary \ref{cor2}.
Moreover, the asymptotic estimate for $B_{n,m}$ in \cite[Prop. 4]{WM} is not correct, because $\exp(e^r)\nsim \exp(n/\log n) $  by \eqref{Wasymp},
even though $e^r =n/r \sim n/\log n$.

\bigskip

We now turn to an application of Theorem \ref{thm1}. In analogy with practical numbers \cite{Sri, PDD} and practical integer partitions \cite{DN,ES}, we say that a partition of a set of $n$ objects is \emph{practical} if its blocks can be combined to form subsets of any size between 1 and $n$. 
Thus, if the partition has $l$ blocks of sizes $a_1, a_2, \ldots ,a_l$,
the partition is practical if and only if
$$ \left\{ \sum_{i=1}^l \varepsilon_i a_i \ : \ \varepsilon_i\in \{0,1\}\right\} = \{0,1,2,3,\ldots,n\}.$$
For example, when $n=7$, the set partition
$\{\{a\},\{b,c\},\{d,e,f,g\}\}$
is practical, but 
$\{\{a\},\{b,c,d\},\{e,f,g\}\}$
is not, because the blocks cannot be combined to form a set of size $2$ or $5$.
Let $P_n$ denote the number of practical set partitions and let $I_n=B_n-P_n$, the number of \emph{impractical}
set partitions, of a set of $n$ elements. Define $P_0=B_0=1$.
Partitions which are $1$-rough are clearly impractical, since the blocks can not be combined to form a set of size $1$. 
Theorem \ref{thm2} shows that, as $n$ grows, almost all impractical set partitions are $1$-rough. 

\begin{theorem}\label{thm2}
We have
$$I_n=B_{n,1}  \left(1+O\left(\exp\left(-\frac{(\log n)^2}{3} \right)\right)\right).$$
\end{theorem}

Combining Theorem \ref{thm2} with Corollary \ref{cor2} yields an estimate for $I_n/B_n$:
\begin{corollary}\label{cor21}
The proportion of impractical set partitions of $n$ objects is
$$ \frac{I_n}{B_n} = \frac{r}{n}\left(1+O\left(\frac{\log n}{n}\right)\right) 
= \frac{\log (n/\log n) }{n}\left(1+O\left(\frac{\log\log n}{(\log n)^2}\right)\right) .$$
\end{corollary}

\begin{table}[h]
  \begin{tabular}{ | c | c |c|c| }
    \hline
    $n$ & $I_n/B_n$ & $r/n$ & Relative Error \\ \hline  
    $2^2$ & 0.533333 &0.300542 &-0.436484 \\ \hline
    $2^4$ & 0.141507 &0.128325 &-0.093156 \\ \hline
    $2^6$ & 0.046743 & 0.047583 & 0.017954  \\ \hline
     $2^8$ & 0.015907 & 0.016123 & 0.013594    \\ \hline
     $2^{10}$ & 0.005122 & 0.005146 & 0.004670  \\ \hline
     \end{tabular}
     \medskip
     
     \caption{Proportion of impractical set partitions: numerical examples of Corollary \ref{cor21}, showing the ratio $I_{n}/B_n$, the approximation $r/n$ and the relative error $(r/n)/(I_n/B_n)-1$.}\label{table2}
\end{table}

Note that the relative errors in Table \ref{table2} are of a similar size as the main term $r/n$,
consistent with the first equation in Corollary \ref{cor21}. 

Since $P_n=B_n-I_n$, we find that almost all set partitions are practical,
as in the case of integer partitions (see \cite{DN,ES}).
\begin{corollary}
The proportion of practical set partitions of $n$ objects is
$$\frac{P_n}{B_n}=1-\frac{r}{n}+O\left(\frac{(\log n)^2}{n^2}\right)
=1-\frac{\log (n/\log n) }{n}+O\left(\frac{\log\log n}{n\log n}\right).$$
\end{corollary}

\section{Proof of Theorem \ref{thm1}}

The following lemma gives a recursive formula for the sequence $B_{n,\mathcal{S}}$, which we used to generate the 
numerical examples in the tables. It is also the basis for deriving the exponential generating function in Lemma \ref{lem2}. 
\begin{lemma}\label{lem1}
For $n\ge 1$,
$$ B_{n,\mathcal{S}} =  \sum_{\substack{0\le k \le  n-1  \\ n-k \notin \mathcal{S}}} \binom{n-1}{k} B_{k,\mathcal{S}}.$$
\end{lemma}
\begin{proof}
We count the number of partitions of the set $\{1,2,\ldots,n\}$, with no block sizes in $\mathcal{S}$. 
For such a partition, let $k$ be the number of all elements of $\{1,2,\ldots,n\}$
which are not in the block that contains $1$. There are  $\binom{n-1}{k}$ ways of selecting those elements
from $\{2,3,\ldots, n\}$, and for each such selection there are $B_{k,\mathcal{S}}$ set partitions of those elements, with
no block size in $\mathcal{S}$.
Note that the block containing $1$ has $n-k$ elements, so $n-k \notin \mathcal{S}$.
\end{proof}

Let
$$G_\mathcal{S}(z)= \sum_{n =0}^\infty \frac{B_{n,\mathcal{S}}}{n!} \, z^n,$$
the exponential generating function for the sequence $B_{n,\mathcal{S}}$.

\begin{lemma}\label{lem2}
We have
$$ G_\mathcal{S}(z)=\exp\bigl(e^z -1 -\alpha(z)\bigr).$$
\end{lemma}
\begin{proof}
It is a standard exercise to derive the differential equation
$$G'_\mathcal{S}(z)=G_\mathcal{S}(z) \left(e^z- \alpha'(z)\right)$$
from the recursive formula in Lemma \ref{lem1}. Solving that equation for $G_\mathcal{S}(z)$ yields the desired result.

Alternatively, the result follows from the general principle in \cite[Proposition II.2]{FS}.
\end{proof}

We will need the following upper bound for the $i$-th derivative $\alpha^{(i)}(r)$:

\begin{lemma}\label{alpha2}
Let $\delta_1, \delta_2$ be as in Theorem \ref{thm1}, and let $I, J \ge 0$ be fixed integers. 
Uniformly for sets $\mathcal{S}$ with $\mathcal{S} \cap [\delta_1 r , \delta_2 r] = \emptyset$, we have 
$$ \alpha^{(i)}(r) \ll_{I,J} \frac{e^{r/2}}{r^J} \qquad (0\le i \le I).$$
\end{lemma}

\begin{proof}
Write $m=\delta_1 r$ and $M=\delta_2 r$. Then 
$$\alpha^{(i)}(r)\le  \sum_{i\le k \le m} \frac{r^{k-i}}{(k-i)!} 
+  \sum_{k > M} \frac{r^{k-i}}{(k-i)!} =s_1+s_2, $$ 
say. We have 
\begin{equation}\label{alphaineq}
s_1 = \sum_{0\le k \le m-i} \frac{m^{k}}{k!}  \left(\frac{r}{m}\right)^{k} 
\le  \left(\frac{r}{m}\right)^{m-i}  \sum_{0\le k \le m-i} \frac{m^k}{k!} <  \left(\frac{r }{m}\right)^{m-i}e^m,
\end{equation}
hence
$$ \log s_1 < m (1-\log (m/r)).$$
Since $m/r=\delta_1 < \eta_1$,
the definition of $\eta_1$ in Theorem \ref{thm1} implies
 $$ (m/r)(1-\log (m/r)) < 1/2-\varepsilon_1,$$
for some $\varepsilon_1 >0$.
Combining the last two inequalities shows that 
$$s_1 < \exp((1/2-\varepsilon_1)r)\ll_J \exp(r/2)/r^J .$$
Similarly,
$$
s_2 = \sum_{k > M-i} \frac{M^{k}}{k!}  \left(\frac{r}{M}\right)^{k} 
\le  \left(\frac{r}{M}\right)^{M-i}  \sum_{k > M-i} \frac{M^k}{k!} < \delta_2^i  \left(\frac{r e}{M}\right)^{M},
$$
hence
$$ \log s_2 < M (1-\log (M/r)) + i \log \delta_2.$$
Since $M/r=\delta_2 > \eta_2$,
the definition of $\eta_2$ in Theorem \ref{thm1} implies
 $$ (M/r)(1-\log (M/r)) < 1/2-\varepsilon_2,$$
for some $\varepsilon_2 >0$.
With $i\le I$, we get 
$$s_2 \ll_I \exp((1/2-\varepsilon_2)r)\ll_{I,J} \exp(r/2)/r^J .$$
\end{proof}

\begin{proof}[Proof of Theorem \ref{thm1}]
Let $r$ be given by \eqref{req}.
Cauchy's residue theorem yields
\begin{equation}\label{Cauchy}
 \frac{B_{n,\mathcal{S}}}{n!} =\frac{1}{2\pi i} \int_{|z|=r} \frac{G_\mathcal{S}(z)}{z^{n+1}} \, dz 
=  \frac{1}{2\pi i} \int_{|z|=r} \frac{\exp\left(e^z -1 -\alpha(z)\right) }{z^{n+1}} \, dz. 
\end{equation}
Writing $z=re^{i\theta}$, we obtain
\begin{equation}\label{BInt}
  \frac{B_{n,\mathcal{S}}}{n!} =\frac{\exp\left(e^r - 1 - \alpha(r)\right)}{2 \pi r^n}
\int_{-\pi}^{\pi} \exp\left(h(\theta) \right)  d \theta,
\end{equation}
where 
$$ h(\theta)= e^{re^{i\theta}}-e^r + \alpha(r)-\alpha\left(r e^{i\theta}\right)-i\theta n.$$

Our first task is to show that the contribution to the last integral from $\delta \le |\theta| \le \pi$ is negligible, where
$$ \delta =\sqrt{12 (1+\alpha(r)) e^{-r}}.$$
We have 
$$ |\exp(h(\theta))| = \exp(\re(h(\theta)))\le \exp( \re e^{re^{i\theta}}-e^r + 2\alpha(r)).$$
Since $\re e^z= e^{\re z}\cos(\im z)$,
$$ \re e^{r(e^{i\theta}-1)} = e^{r(\cos \theta -1)}\cos(r\sin\theta)\le e^{r(\cos \theta -1)}\le e^{r(\cos \delta -1)},$$
for $\delta \le |\theta| \le \pi$.
Now $\cos \delta \le 1-\delta^2/3$ for $0\le \delta \le 2$, and $e^x \le 1+x/2$ for $-1\le x \le 0$. Thus
$$ \re e^{r(e^{i\theta}-1)} \le 1+ r(\cos \delta -1)/2 \le 1-\frac{r\delta^2}{6},$$
and
$$ \re  e^{re^{i\theta}}-e^r = e^r \left( \re e^{r(e^{i\theta}-1)} -1\right)\le - \frac{e^r r\delta^2}{6}=-2r(1+\alpha(r)), $$
for $\delta \le |\theta| \le \pi$.
Hence
$$ |\exp(h(\theta))| \le \exp\bigl(-2r(1+\alpha(r))+2\alpha(r)\bigr)\ll \exp(-2r).$$
The contribution to the integral in \eqref{BInt} from $\delta \le |\theta| \le \pi$ is thus 
\begin{equation}\label{tails}
\left| \int_{\delta \le |\theta|\le \pi}  \exp\left(h(\theta) \right)  d \theta \right|
  \le \int_{\delta \le |\theta|\le \pi} \left| \exp\left(h(\theta) \right) \right|  d \theta
\ll e^{-2r},
\end{equation}
which is acceptable.

The second task is to approximate $h(\theta)$ for $|\theta|\le \delta$ by a Taylor polynomial and show that the error term is negligible.
We have
$$h(\theta)=-i \alpha'(r) r \theta - A\frac{\theta^2}{2} - iB \frac{\theta^3}{6}+O(r^4 e^r \theta^4),$$
where
$$A=r^2 e^r(1-\alpha''(r)e^{-r})+ r e^r(1-\alpha'(r)e^{-r}),$$
and 
$$B=e^r(r^3+3r^2+r)-(r^3\alpha'''(r)+3r^2\alpha''(r)+r\alpha'(r))=O(e^r r^3).$$
Since $e^{-it}=1-it+O(t^2)$ for all real $t$, we can write $\exp(h(\theta))$ as
$$e^{- A\theta^2/2} 
\left(1-   i \alpha'(r) r \theta -  iB \theta^3/6 +O((r\alpha'(r))^2 \theta^2+B^{2} \theta^6)   \right)
\left(1+O(r^4 e^r \theta^4)\right),
$$
where the last error term is justified since $r^4 e^r \theta^4=O(1)$ for $|\theta|\le \delta$, by Lemma \ref{alpha2}.
Multiplying the two factors, and appealing to $r^4 e^r \theta^4=O(1)$, shows that $\exp(h(\theta))$ equals
$$
e^{- A\theta^2/2} \left(1-   i \alpha'(r) r \theta -  iB \theta^3/6 
+O\left((r\alpha'(r))^2  \theta^2 + B^2 \theta^6 + r^4 e^r \theta^4 \right)\right)
$$
and 
$$ \int_{-\delta}^{\delta} \exp\left(h(\theta) \right)  d \theta
=\int_{-\delta}^{\delta} e^{- A\theta^2/2} d \theta + 0 + 0 +E,$$
where
$$
E\ll \int_{-\delta}^{\delta} e^{- A\theta^2/2} \left((r\alpha'(r))^2  \theta^2 + B^2 \theta^6 + r^4 e^r \theta^4 \right)d \theta .
$$
The even central moments of a normal distribution with variance $1/A$ are given by (see \cite[p. 25]{HB})
$$ \sqrt{\frac{A}{2 \pi}}  \int_{-\infty}^{\infty}  e^{- A\theta^2/2} \theta^{2k} d \theta
=\frac{(2k)!}{(2A)^k k!} \qquad (k\ge 0).
$$
Since
  $A \gg e^r r^2$, by Lemma \ref{alpha2}, and $B\ll e^r r^3$, we obtain
$$
E\sqrt{A} \ll \frac{(r\alpha'(r))^2  }{A} + \frac{B^2}{A^3} + \frac{r^{4} e^r}{A^2} 
\ll (\alpha'(r))^2 e^{-r}+ e^{-r}+e^{-r}
$$
and 
$$  \int_{-\delta}^{\delta} \exp\left(h(\theta) \right)  d \theta =\int_{-\delta}^{\delta} e^{- A\theta^2/2} d \theta +
O\left(\frac{1+(\alpha'(r))^2}{\sqrt{A}e^r}\right).
$$

Our third task is to extend the last integral to $(-\infty, \infty)$. We have
$$
\int_{\delta}^{\infty} e^{- A\theta^2/2} d \theta
\le \int_{\delta}^{1} e^{- A\theta^2/2} d \theta + \int_{1}^{\infty} e^{- A\theta/2} d \theta
\le  e^{- A\delta^2/2} + \frac{ e^{- A/2}}{A/2} \ll e^{-2r}.
$$
Thus,
\begin{equation}\label{central}
\begin{split}
 \int_{-\delta}^{\delta} \exp\left(h(\theta) \right)  d \theta 
 &=\int_{-\infty}^{\infty} e^{- A\theta^2/2} d \theta +O\left(\frac{1+(\alpha'(r))^2}{\sqrt{A}e^r}\right)\\
 &= \sqrt{\frac{2\pi}{A}}\left(1+O\left(\frac{1+(\alpha'(r))^2}{e^r}\right)\right).
\end{split}
\end{equation}

To approximate the quantity $A$ by $r(r+1)e^r $, we need to estimate $\alpha''(r)$ in terms of $\alpha'(r)$. We have
$$  \sum_{\substack{1\le k-1 \le 3r \\ k \in \mathcal{S}}} \frac{r^{k-2}}{(k-2)!} 
\le  \sum_{\substack{1\le k-1 \le 3r \\ k \in \mathcal{S}}} \frac{r^{k-2}}{(k-2)!} \cdot \frac{3r}{k-1} 
\le 3\alpha'(r)
$$
and
$$
 \sum_{\substack{k-1 > 3r \\ k \in \mathcal{S}}} \frac{r^{k-2}}{(k-2)!} 
 \le  \sum_{\substack{k-1 > 3r \\ k \in \mathcal{S}}} \frac{(3r)^{k-2}}{(k-2)!} \left(\frac{1}{3}\right)^{3r-1} 
 < e^{3r}  \left(\frac{1}{3}\right)^{3r-1} \le 3.
$$
Thus, $\alpha''(r) \le 3 + 3\alpha'(r)$, and 
\begin{equation}\label{Aestimate}
  A=r(r+1)e^r \left(1+O\left(\frac{1+\alpha'(r)}{e^r} \right)\right).
\end{equation}
Theorem \ref{thm1} now follows from combining the estimates \eqref{tails}, \eqref{central} and \eqref{Aestimate} with equation \eqref{BInt}.
\end{proof}

\section{Proof of Theorem \ref{thm2}}

Sierpinski \cite{Sier} and Stewart \cite{Stew} independently gave the characterization of practical numbers in terms of their prime factors.
The following analogue characterizes practical set partitions in terms of their block sizes. 

\begin{lemma}\label{lp0}
A set partition with $l$ blocks of sizes $a_1 \le a_2 \le \ldots \le a_l$
is practical if and only if 
\begin{equation}\label{condp}
a_i \le 1+ \sum_{1\le j < i} a_j  \qquad (1\le i \le l).
\end{equation}
\end{lemma}

\begin{proof}
Condition \eqref{condp} is clearly necessary: if $a_i > 1+ \sum_{1\le j < i} a_j$ for some $1\le i \le l$, 
then there is no set of size $1+ \sum_{1\le j < i} a_j$ which is the union of different blocks.

To show that \eqref{condp} is sufficient, we proceed by induction on $l$. The case $l=1$ is obvious.
Assume that \eqref{condp} implies that the corresponding set partition 
with block sizes $a_1\le \ldots \le a_l$ is practical for some $l\ge 1$. 
Assume a set partition with block sizes $a_1\le \ldots \le a_l \le a_{l+1}$ satisfies \eqref{condp}, with $l$ replaced by $l+1$. 
The set of sizes of subsets obtained from combining different blocks is 
$$ A:=\left\{ \sum_{i=1}^{l+1} \varepsilon_i a_i \ : \ \varepsilon_i\in \{0,1\}\right\} 
=  \left\{ \sum_{i=1}^{l} \varepsilon_i a_i \ : \ \varepsilon_i\in \{0,1\}\right\} + \left\{ 0 , a_{l+1}\right\}$$
By the inductive hypothesis,  
$ A=\left\{1, 2, 3, \ldots,  \sum_{i=1}^{l} a_i \right\}+\left\{ 0 , a_{l+1}\right\},$
and since $a_{l+1} \le 1+ \sum_{1\le j < l+1} a_j$, we have
$   A=\left\{1, 2, 3, \ldots,  \sum_{i=1}^{l+1} a_i \right\}.$
\end{proof}

The following functional equation is the analogue of \cite[Lemma 2.3]{PDD} for practical numbers and of \cite[Lemma 5]{DN} for practical integer partitions.
Other analogues include polynomials over finite fields \cite[Lemma 5]{DPD} and permutations \cite[Lemma 11]{DPD}.

\begin{lemma}\label{lp1}
For $n\ge 0$,
$$B_n = \sum_{k=0}^n \binom{n}{k} P_k B_{n-k,k+1}.$$
\end{lemma}

\begin{proof}
Given any partition of $n$ objects with block sizes $a_1\le a_2 \le \ldots \le a_l$, let $l_0$ be the largest index such that
\begin{equation*}
a_i \le 1+ \sum_{1\le j < i} a_j  \qquad (1\le i \le l_0),
\end{equation*}
and let $l_0=0$ if $a_1>1$. By Lemma \ref{lp0}, the blocks of sizes $a_1\le \ldots \le a_{l_0}$ form a practical set partition of 
a set with $k:= \sum_{1\le j \le l_0} a_j$ elements. 
Since $l_0$ was maximal, the remaining blocks have sizes $k+1<a_{l_0+1}\le \ldots \le a_l$ and 
form a $(k+1)$-rough set partition of a set with $n-k$ elements. The lemma now follows since $0\le k \le n$ and there are
$ \binom{n}{k}$ ways to choose the $k$ elements that belong to the practical set partition.
\end{proof}

\begin{lemma}\label{lp2}
For $n\ge 1$,
$$ I_n = B_{n,1} + \sum_{k=1}^{\lfloor \frac{n-2}{2}\rfloor}  \binom{n}{k} P_k B_{n-k,k+1}.
$$
\end{lemma}

\begin{proof}
In Lemma \ref{lp1}, the term corresponding to $k=0$ is $B_{n,1}$, since $P_0=1$. The term corresponding to $k=n$ is $P_n$, since 
$B_{0,n+1}=1$. If $(n-2)/2<k<n$, then $0<n-k < k+2$, so $B_{n-k,k+1}=0$. The result now follows since $I_n=B_n-P_n$.
\end{proof}

For the remainder of this section, we write
$$ \beta_m(z) = \sum_{j=0}^m \frac{z^j}{j!}.$$

\begin{lemma}\label{lp3}
For $n\ge 1$, $m\ge 0$, we have
$$ B_{n,m} \le  \frac{n! \exp\bigl(e^r-\beta_m(r)\bigr)}{ r^n } $$
\end{lemma}

\begin{proof}
With $\mathcal{S}=\{1,2,3,\ldots, m\}$, the integrand in equation \eqref{Cauchy} satisfies
$$ \left| \frac{\exp\left(e^z-\beta_m(z)\right) }{z^{n+1}} \right|=
 \frac{\exp\left(\re \sum_{k=m+1}^\infty \frac{z^k}{k!}\right) }{|z|^{n+1}} 
 \le  \frac{\exp\left(\sum_{k=m+1}^\infty \frac{r^k}{k!}\right) }{r^{n+1}} ,
$$
and therefore
$$  \frac{B_{n,m}}{n!} \le \frac{1}{2\pi }\, 2\pi r \, \frac{\exp\left(\sum_{k=m+1}^\infty \frac{r^k}{k!}\right) }{r^{n+1}} 
= \frac{ \exp\bigl(e^r-\beta_m(r)\bigr)}{ r^n }.
$$
\end{proof}

\begin{proof}[Proof of Theorem \ref{thm2}]
We write $b_{n,k}=B_{n,k}/n!$, $p_n=P_n/n!$ and $i_n=I_n/n!$. Lemma \ref{lp2} says that 
$$ 0 \le  i_n - b_{n,1} = \sum_{k=1}^{N}  p_k b_{n-k,k+1} \le  \sum_{k=1}^{N}  b_k b_{n-k,k+1} ,
$$
where $N=\lfloor \frac{n-2}{2}\rfloor$. To establish Theorem \ref{thm2}, 
we need to show that the last sum satisfies
\begin{equation}\label{tbp}
 \sum_{k=1}^{N}  b_k b_{n-k,k+1} \ll b_{n,1} \exp(-(\log n)^2/3). 
\end{equation}
We write
$$  \sum_{k=1}^{N}  b_k b_{n-k,k+1} =   \sum_{1\le k \le L}+  \sum_{L<k\le M}+  \sum_{M<k\le N}=S_1+S_2+S_3,$$
say, where 
$ L= (\log n)^{3/2}$ and $M=\lfloor \frac{n-4}{3}\rfloor$. 
Let
$$ g(n) =n \log(r(n))-\frac{n}{r(n)}\sim n\log \log n,$$
by \eqref{Wasymp}. 
It is easy to verify that
\begin{equation}\label{gprime}
g'(n)=\log r(n) \sim \log \log n.
\end{equation}
Theorem \ref{thm1} shows that
\begin{equation*}
  b_{n,1} =\exp(e^r-n\log r +O(r)) =\exp(-g(n) +O(\log n)),
\end{equation*}
since $e^r=n/r$. Similarly, Lemma \ref{lp3} shows that, for $n\ge 1$ and $m\ge 0$,
$$ b_{n,m} \le \exp(-g(n) -\beta_m(r(n))).$$
It follows that
\begin{equation}\label{lbb}
 \log( b_k b_{n-k,k+1}) \le - g(k) - g(n-k) - \beta_{k+1}(r(n-k)). 
\end{equation}
For $S_1$, we have $1\le k\le L=(\log n)^{3/2}$, so \eqref{lbb} and \eqref{gprime} yield
\begin{equation*}
\begin{split}
  \log( b_k b_{n-k,k+1}) &\le 0 -g(n-L) -\beta_2(r(n/2))\\
  & \le - g(n) +O(L \log\log n) - (\log n)^2/(2+\varepsilon) \\
  & \le  - g(n)  - (\log n)^2/(2+2\varepsilon),
\end{split}
\end{equation*}
for any $\varepsilon >0$ and $n\ge n_0(\varepsilon).$
Hence
\begin{equation}\label{S1B}
\begin{split}
S_1 & \le L \exp(- g(n)  - (\log n)^2/(2+2\varepsilon)) \\
& \le b_{n,1} \exp(- (\log n)^2/(2+3\varepsilon)).
\end{split}
\end{equation}

When $L<k\le N$, we have
\begin{equation*}
\begin{split}
  0\le e^{r(n-k)}-\beta_{k+1}(r(n-k)) &= \sum_{j=k+2}^\infty \frac{(r(n-k))^j}{j!} 
 \le  \sum_{j=k+2}^\infty \frac{k^j}{j!}\left(\frac{r(n)}{k}\right)^j \\
&\le \left(\frac{r(n)}{L}\right)^{k+2} e^k = o (1), 
\end{split}
\end{equation*}
as $n\to \infty$. Since $ e^{r(n-k)}=(n-k)/r(n-k)$, \eqref{lbb} yields
\begin{equation*}
\begin{split}
  \log( b_k b_{n-k,k+1}) &\le -g(k) - g(n-k) -\frac{n-k}{r(n-k)} +o(1) \\
  & = -f(n,k)+o(1),
\end{split}
\end{equation*}
say. We claim that $f(n,k)$ is decreasing (and hence $-f(n,k)$ is increasing) in $k$, for $1\le k \le n/2$.
Indeed, since $r(n)$ is increasing, \eqref{gprime} shows that $g(k)+g(n-k)$ is decreasing in $k$, for $1\le k \le n/2$.
Moreover, $ e^{r(n-k)}$ is clearly decreasing in $k$. 
Thus,
\begin{equation}\label{S2B}
\begin{split}
S_2 & \le M \exp(-f(n,M) +o(1)) \\
& \le b_{n,1} \exp(g(n) - f(n,n/3) + O(\log n)) \\
& \le  b_{n,1} \exp(-0.03 n/ \log n),
\end{split}
\end{equation}
for $n$ sufficiently large. 
The last inequality, whose derivation is not difficult but somewhat tedious, follows from \eqref{Wasymp}.

Finally, when $M<k\le N$, we have $n-k\ge k+2 > (n-k)/2$. Thus, $B_{n-k,k+1}=1$ and $b_{n-k,k+1}=1/(n-k)!$.
Stirling's approximation, in the form $\log(n!)=n(\log n -1) +O(\log n)$, yields 
\begin{equation*}
\begin{split}
  \log( b_k b_{n-k,k+1}) &\le - g(k) - \log( (n-k)!)\\
  & \le 0-(n-k)(\log(n-k) -1) +O(\log n) \\
  & \le - (n/2)(\log(n/2)-1)+O(\log n).
\end{split}
\end{equation*}
Hence
\begin{equation}\label{S3B}
\begin{split}
S_3 & \le N \exp( - (n/2)(\log(n/2)-1)+O(\log n)) \\
& \le b_{n,1} \exp(g(n)   - (n/2)(\log(n/2)-1) + O(\log n)) \\
& \le  b_{n,1} \exp\left(- \frac{n \log n}{3} \right),
\end{split}
\end{equation}
for $n$ sufficiently large. 

The estimates \eqref{S1B}, \eqref{S2B} and \eqref{S3B} show that 
\eqref{tbp} holds, which completes the proof of Theorem \ref{thm2}.
\end{proof}


\begin{thebibliography}{99}

\bibitem{CGH}
 R. M. Corless,  G. H. Gonnet,  D. E. G. Hare, D. J. Jeffrey, D. E. Knuth, 
On the Lambert $W$ function,
Adv. Comput. Math. \textbf{5} (1996), no. 4, 329--359. 

\bibitem{DN}
J. Dixmier and J.-L. Nicolas, Partitions without small parts. Number theory, Vol. I (Budapest, 1987), 9--33, 
Colloq. Math. Soc. J\'anos Bolyai, 51, North-Holland, Amsterdam, 1990. 

\bibitem{ES}
P. Erd\H{o}s and M. Szalay, On some problems of J. D\'enes and P. Tur\'an, Studies in pure mathematics, 187--212, Birkh\"auser, Basel, 1983. 

\bibitem{FS}
P. Flajolet and R. Sedgewick, Analytic Combinatorics, Cambridge Univ. Press, 2009.

\bibitem{HB}
J. Patel and C. Read, Handbook of the Normal Distribution, Second Edition, CRC Press, 1996.

\bibitem{MW}
L. Moser and M. Wyman, An asymptotic formula for the Bell numbers,
Transactions of the Royal Society of Canada {\bf 49} (1955) 49-54.

\bibitem{Sier}
W. Sierpinski, Sur une propri\'et\'e des nombres naturels,
\textit{Ann. Mat. Pura Appl. (4)} \textbf{39} (1955), 69--74.

\bibitem{Sri}
A. K. Srinivasan, 
Practical numbers, \textit{Current Sci.} \textbf{17} (1948), 179--180.

\bibitem{Stew}
B. M. Stewart, Sums of distinct divisors, \textit{Amer. J. Math.} \textbf{76} (1954), 779--785. 

\bibitem{WM}
C. Wang and I. Mez\H{o},
Some limit theorems with respect to constrained permutations and partitions, 
\textit{Monatsh. Math.} \textbf{182} (2017), no. 1, 155--164. 

\bibitem{PDD}
A. Weingartner, Practical numbers and the distribution of divisors, Q. J. Math.  \textbf{66} (2015), 743--758.

\bibitem{DPD}
A. Weingartner, On the degrees of polynomial divisors over finite fields, Math. Proc. Cambridge Philos. Soc. \textbf{161} (2016), no. 3, 469--487.

\end{thebibliography}
\end{document}